\documentclass[a4paper,reqno]{amsart}

\usepackage{tikz}
\usetikzlibrary{matrix,arrows}

\usepackage{amssymb}
\usepackage{amstext}
\usepackage{amsmath,mathtools}
\usepackage{amscd}
\usepackage{amsthm}
\usepackage{amsfonts}

\usepackage{enumerate}
\usepackage{mathrsfs}

\usepackage[all]{xy}
\xyoption{all}

\newtheorem{theorem}{Theorem}[section]

\newtheorem{corollary}[theorem]{Corollary}
\newtheorem{lemma}[theorem]{Lemma}
\newtheorem{proposition}[theorem]{Proposition}
\newtheorem{definition-proposition}[theorem]{Definition-Proposition}

\theoremstyle{definition}
\newtheorem{definition}[theorem]{Definition}

\newtheorem{remark}[theorem]{Remark}
\newtheorem{example}[theorem]{Example}

\newcommand{\End}{\operatorname{End}\nolimits}

\newcommand{\Rad}{\operatorname{Rad}\nolimits}

\newcommand{\V}{\mathcal{V}}
\newcommand{\X}{\mathcal{X}}
\newcommand{\Y}{\mathcal{Y}}
\newcommand{\E}{\mathbb{E}}
\newcommand{\MM}{\mathscr{M}}
\newcommand{\WW}{\mathscr{W}}

\DeclareMathOperator{\eHom}{\preccurlyeq}
\DeclareMathOperator{\noteHom}{\not \preccurlyeq}
\DeclareMathOperator{\eExt}{\curlyeqprec}
\newcommand{\wide}{\operatorname{wide}\nolimits}

\DeclareMathOperator{\Hom}{Hom}
\DeclareMathOperator{\Ext}{Ext}

\renewcommand{\mod}{\operatorname{mod}}

\DeclareMathOperator{\add}{add}

\begin{document}
\title[Wide subcategories for higher Auslander algebras]{Classification of higher wide subcategories for higher Auslander algebras of type $A$}

\author{Martin Herschend}
\address{Department of Mathematics, Uppsala University, Box 480, 751 06 Uppsala, Sweden}
\email{martin.herschend@math.uu.se}

\author{Peter J\o rgensen}
\address{Department of Mathematics, Aarhus University, 8000 Aarchs C, Denmark}
\email{peter.jorgensen@math.au.uk}
\urladdr{https://sites.google.com/view/peterjorgensen}

\thanks{The first author was supported by the Swedish Research Council grant number 621-2014-3983 and the second author by EPSRC grant EP/P016014/1 ``Higher Dimensional Homological Algebra''.}

\keywords{Higher Auslander algebra, $d$-abelian category, $d$-cluster tilting subcategory, higher homological algebra, wide subcategory}  

\begin{abstract}
A subcategory $\WW$ of an abelian category is called wide if it is closed under kernels, cokernels, and extensions.  Wide subcategories are of interest in representation theory because of their links to other homological and combinatorial objects, established among others by Ingalls--Thomas and Marks--\v{S}\v{t}ov\'{\i}\v{c}ek.

If $d \geqslant 1$ is an integer, then Jasso introduced the notion of $d$-abelian categories, where kernels, cokernels, and extensions have been replaced by longer complexes.  Wide subcategories can be generalised to this situation.

Important examples of $d$-abelian categories arise as the $d$-cluster tilting subcategories $\MM_{ n,d }$ of $\mod A_n^{ d-1 }$, where $A_n^{ d-1 }$ is a higher Auslander algebra of type $A$ in the sense of Iyama.  This paper gives a combinatorial description of the wide subcategories of $\MM_{ n,d }$ in terms of what we call non-interlacing collections.  
\end{abstract}

\maketitle

\tableofcontents

\section{Introduction}

There has recently been considerable interest in wide subcategories of abelian categories, which are full subcategories closed under kernels, cokernels, and extensions.  The theory has been developed to a high level, and deep links with several other areas of mathematics have been revealed.  In many cases, there are bijections from the set of wide subcategories to other classes of mathematical objects like cluster tilting objects, non-crossing partitions, support varieties, and torsion classes.  Algebraic and representation theoretical aspects of the theory have been investigated in \cite{B}, \cite{F}, \cite{HJV}, \cite{MMSS}, \cite{MS}, \cite{R}, \cite{Y}, geometrical aspects in \cite{BIK}, \cite{H}, and combinatorial aspects in \cite{Br}, \cite{GG}, \cite{GS}, \cite{HK}, \cite{IPT}, \cite{IT}.  

For example, it was shown in \cite[thm.\ 1.1]{IT} that in the category of finite dimensional representations of a finite, acyclic quiver $Q$, there is a bijection between wide subcategories and torsion classes.  There is also a bijection to basic cluster tilting objects in the cluster category associated with $Q$, and if $Q$ is of extended Dynkin type, then there is a bijection to the noncrossing partitions associated with $Q$.  The bijection between wide subcategories and torsion classes was extended to a large class of module categories over finite dimensional algebras in \cite[cor.\ 3.11]{MS}. 

For an integer $d \geqslant 1$, the notion of $d$-abelian categories was introduced in \cite{J} as an analogue of abelian categories from the point of view of higher dimensional Auslander-Reiten theory. In a $d$-abelian category, kernels, cokernels and extensions are replaced by longer complexes, see \cite[def.\ 3.1]{J}.  The notion of wide subcategories was generalised to $d$-abelian categories in \cite[def.\ 2.11]{HJV}. A general theory of such wide subcategories was developed in \cite{HJV}, and some simple examples were worked out in \cite[sec.\ 7]{HJV}.  

The prototypical example of a $d$-abelian category is a $d$-cluster tilting subcategory of the module category of a finite dimensional algebra. Algebras of global dimension $d$ admitting $d$-cluster tilting subcategories with finitely many indecomposables are called $d$-representation finite \cite{IO} and have been extensively studied as a source of $d$-abelian categories. One of the features of a $d$-representation finite algebra is that its module category has a unique $d$-cluster tilting subcategory.

An algebra is $1$-representation finite if and only if it is representation finite and hereditary. So if $Q$ is a Dynkin quiver, then the path algebra of $Q$ is $1$-representation finite. For $d>1$, the first known $d$-representation finite algebras are the higher Auslander algebras of type $A$ introduced by Iyama in \cite{I2}. They are constructed recursively starting from path algebras of Dynkin type $A$ with linear orientation. Then for each $d$, one takes the algebras obtained for $d-1$ and passes to their corresponding higher Auslander algebras. The structure of the corresponding $d$-abelian categories was described in detail in \cite{OT}. Recently, higher Auslander algebras of type $A$ have been used in \cite{DJL} as a tool to study partially wrapped Fukaya categories of symmetric products of the unit disc with finitely many stops.

In this paper we describe combinatorially all wide subcategories of the $d$-abelian categories which arise from higher Auslander algebras of type $A$.  Our main tools are the methods of \cite{HJV} and \cite{OT}.  The main result is Theorem \ref{main}, and the description is in terms of what we call non-interlacing collections.  It has as a special case the bijection in \cite[thm.\ 1.1]{IT} between classic wide subcategories and non-crossing partitions in Dynkin type $A$.

The $d$-cluster tilting subcategory $\MM$ of the module category of a $d$-representation finite algebra also gives rise to a $d$-cluster tilting subcategory $\mathscr{C}$ of its bounded derived category. The category $\mathscr{C}$ has the structure of a $(d+2)$-angulated category in the sense of \cite{GKO}, and there is a strong interplay between the $d$-abelian structure of $\MM$ and the $(d+2)$-angulated structure of $\mathscr{C}$. In certain situations this can even be used to study $(d+2)$-angulated categories in a more general setting (see for instance \cite{JJ}).

In \cite{B} it is shown that for a hereditary algebra $A$ there is a bijection between classic wide subcategories of the module category of $A$ and wide subcategories of its bounded derived category. This was generalised in \cite{F} to the setting described above, i.e., there is a bijection between the wide subcategories of $\MM$ and the wide subcategories of $\mathscr{C}$. In particular, this can be applied to higher Auslander algebras of type $A$. In other words, combining our results with \cite{F} we obtain a classification of wide subcategories of the $(d+2)$-angulated categories associated to higher Auslander algebras of type $A$.

\section{Preliminaries}

\subsection{Conventions}
Let $K$ be a field. All categories and functors are assumed to be $K$-linear. In all contexts we denote the $K$-dual $\Hom_K(-,K)$ by $D$.

Let $\MM$ be a category and $M, N \in \MM$. A diagram $M \leftrightarrow \cdots \leftrightarrow N$, where $\leftrightarrow$ represents a non-zero morphism in either direction is called a walk from $M$ to $N$. We call a category connected if any two non-zero objects are connected by a walk.

Let $\MM$ be an additive category. By an additive subcategory $\WW \subseteq \MM$ we mean a full subcategory closed under direct sums and summands. For a collection of objects $\mathcal{C}$ in $\MM$ we denote by $\add \mathcal{C}$ the smallest additive subcategory of $\MM$ containing $\mathcal{C}$.

We call $\MM$ Krull-Schmidt if each object decomposes into a direct sum of finitely many indecomposable objects and each indecomposable object has local endomorphism algebra. For such a category we denote by $\Rad_\MM$ its Jacobson radical, i.e.,
\[
\Rad_\MM(M,N) = \{f \in \MM(M,N) \mid 1_M - g \circ f \mbox{ is invertible for all } g \in \MM(N,M)\}.
\]
The square of the Jacobson radical is
\[
\Rad^2_\MM(M,N) = \{f\circ g \mid f \in \Rad_\MM(U,N),\, g \in \Rad_\MM(M,U) \mbox{ for some } U \in \MM\}.
\]
Next assume that $\MM$ is skeletally small, Hom-finite and $\MM(M,M)/\Rad_\MM(M,M) = K$ for all indecomposable $M \in \MM$. We then define the quiver $Q$ of $\MM$ as follows. As vertices $Q_0$ choose a set of representatives of the isomorphism classes of indecomposable objects in $\MM$. For $M,N \in Q_0$ we choose as arrows form $M$ to $N$ a subset of $\Rad_\MM(M,N)$, that gives a $K$-basis of $\Rad_\MM(M,N)/\Rad^2_\MM(M,N)$. 

By an algebra we mean a finite dimensional $K$-algebra. For an algebra $A$ we denote the category of finitely generated right $A$-modules by $\mod A$. By $A$-module we always mean an object in $\mod A$. We use the terminology of quivers, quiver representations, and path algebras following the conventions in \cite{ASS}. 

Let $\WW$ be an additive subcategory of $\mod A$. By a $\WW$-resolution of $M \in \mod A$ we mean an exact sequence
\[
0 \to W_m \to \cdots \to W_0 \to M \to 0,
\]
where $m \geqslant 0$ and $W_i \in \WW$.

\subsection{$d$-abelian categories and wide subcategories}
In this section we introduce some basic results for $d$-abelian categories. We mainly rely on \cite{J}, in which $d$-abelian categories were first introduced. We then recall the notion of a wide subcategory of a $d$-abelian category following \cite{HJV}.

\begin{definition}\label{dAbel}
Let $\MM$ be an additive category and
\[
\E: \quad M_{d+1} \overset f \to M_{d} \to \cdots \to M_{1} \overset g\to M_{0}
\]
a sequence in $\MM$. 
\begin{enumerate}
\item We call 
\[
M_{d+1} \overset f \to M_{d} \to \cdots \to M_{1}
\]
a \emph{$d$-kernel} of $g$ if
\[
0 \to \MM(M,M_{d+1}) \overset {f\circ -} \to \MM(M,M_{d}) \to \cdots \to \MM(M,M_{0})
\]
is exact for all $M \in \MM$.
\item We call 
\[
M_{d} \to \cdots \to M_{1} \overset g\to M_{0}
\]
a \emph{$d$-cokernel} of $f$ if
\[
0 \to \MM(M_{0},M) \overset {- \circ g} \to \MM(M_{1},M) \to \cdots \to \MM(M_{d+1},M)
\]
is exact for all $M \in \MM$.
\item If both (1) and (2) are satisfied we call $\E$ a \emph{$d$-exact} sequence (or a \emph{$d$-extension} of $M_{0}$ by $M_{d+1}$).
\item We say that $\MM$ is \emph{$d$-abelian} if it is idempotent split, every morphism admits a $d$-kernel and $d$-cokernel, and every monomorphism $f$ respectively epimorphism $g$ fits into a $d$-exact sequence of the form $\E$.
\end{enumerate}
\end{definition}

There is a natural equivalence relation on $d$-extensions introduced in \cite{J}, which we now recall.

\begin{definition}\label{def.extensions}
Let $\MM$ be a $d$-abelian category. We call two $d$-extensions
\[
\E: \quad X \to E_{d} \to \cdots \to E_{1} \to Y
\]
and
\[
\E': \quad X \to E'_{d} \to \cdots \to E'_{1} \to Y
\]
equivalent if there is a commutative diagram
\[
\xymatrix{
X\ar@{=}[d] \ar[r] & E_{d} \ar[r] \ar[d] &  \cdots \ar[r] &  E_{1} \ar[r]\ar[d] &  Y \ar@{=}[d] \\
X \ar[r] & E'_{d} \ar[r] &  \cdots \ar[r] &  E'_{1} \ar[r] &  Y
}
\]
It follows from \cite[Proposition 4.10]{J} that this indeed does define an equivalence relation. 
\end{definition}

Next we introduce the notion of wide subcategories.
\begin{definition}\label{def.wide}
Let $\MM$ be a $d$-abelian category. We call an additive subcategory $\WW \subseteq \MM$ \emph{wide} if the following conditions hold.
\begin{enumerate}
\item Every morphism $f: M \to N$ in $\WW$ admits a $d$-kernel and $d$-cokernel in $\MM$ with terms in $\WW$.
\item Every $d$-extension
\[
\E: \quad X \to E_{d} \to \cdots \to E_{1} \to Y
\]
in $\MM$ with $X,Y \in \WW$ is equivalent to a $d$-extension
\[
\E: \quad X \to E'_{d} \to \cdots \to E'_{1} \to Y
\]
with $E'_{i} \in \WW$ for all $1 \leqslant i \leqslant d$.
\end{enumerate}
\end{definition}
For a class of objects $\mathcal{C}$ in $\MM$ we denote the smallest wide subcategory of $\MM$ containing $\mathcal{C}$ by $\wide\mathcal{C}$.

A difference between abelian and $d$-abelian categories is that $d$-kernels and $d$-cokernels are not unique up to isomorphism. Similarly, equivalent $d$-extensions may be non-isomorphic. However, it is shown in \cite{J} that uniqueness holds if we replace isomorphism of complexes by homotopy equivalence. 

Next we consider the case when $\MM$ is a Krull-Schmidt category. We start by showing that representatives of $d$-kernels, $d$-cokernels and $d$-extensions can be chosen in a certain minimal way that is unique.

\begin{proposition}\label{radical}
Let $\MM$ be a $d$-abelian Krull-Schmidt category and $X,Y \in \MM$. 
\begin{enumerate}[\rm(1)]
\item Let $f \in \MM(X,Y)$. Then there is a $d$-kernel
\[
K_{d} \overset{k_d}\to \cdots \overset{k_2} \to K_{1} \to X
\]
of $f$ such that $k_i \in \Rad_\MM(K_i, K_{i-1})$ for all $2 \leqslant i \leqslant d$. Moreover, this $d$-kernel appears as a direct summand (in the category of $\MM$-complexes) of any other $d$-kernel of $f$. 
\item Let $f \in \MM(X,Y)$. Then there is a $d$-cokernel
\[
Y \to C_{d} \overset{c_d}\to \cdots \overset{c_2} \to C_{1}
\]
of $f$ such that $c_i \in \Rad_\MM(C_i, C_{i-1})$ for all $2 \leqslant i \leqslant d$. Moreover, this $d$-cokernel appears as a direct summand (in the category of $\MM$-complexes) of any other $d$-cokernel of $f$. 
\item In every equivalence class of $d$-extensions of $Y$ by $X$ there is a representative
\[
\E: \quad X \to E_{d} \overset{e_{d}}\to \cdots \overset{e_2}\to E_{1} \to Y
\]
such that $e_i \in \Rad_\MM(E_i, E_{i-1})$ for all $2 \leqslant i \leqslant d$. Moreover, $\E$ is a direct summand (in the category of $\MM$-complexes) of every other equivalent $d$-extension.
\end{enumerate}
\end{proposition}
\begin{proof}
(1) Let 
\[
K_{d} \overset{k_d}\to \cdots \overset{k_2} \to K_{1} \overset{k_1}\to X
\]
be a $d$-kernel of $f$ such that the number of indecomposable direct summands of $\bigoplus_{i=1}^d K_i$ is minimal. If $k_i$ is not a radical morphism for some $2 \leqslant i \leqslant d$, then we may write $K_i = U \oplus \tilde K_{i}$ and $K_{i-1} = U \oplus \tilde K_{i-1}$ such that $k_i = 1_U \oplus \tilde k_{i}$ for some $\tilde k_{i} : \tilde K_{i} \to \tilde K_{i-1}$ where $U$ is indecomposable. It follows that replacing $k_i : K_i \to K_{i-1}$ in the $d$-kernel by $\tilde k_{i} : \tilde K_{i} \to \tilde K_{i-1}$, gives a $d$-kernel with smaller total number of indecomposable direct summands in its terms, which contradicts minimality.

Next let 
\[
K'_{d} \overset{k'_d}\to \cdots \overset{k'_2} \to K'_{1} \overset{k'_1}\to X
\]
be another $d$-kernel of $f$. Using the defining property of $d$-kernels we obtain the commutative diagrams
\[
\xymatrix{
K_d\ar[d]^{a_d} \ar[r]^{k_d} & K_{d-1} \ar[r]^{k_{d-1}} \ar[d]^{a_{d-1}} &  \cdots \ar[r]^{k_2} &  K_{1} \ar[r]^{k_1}\ar[d]^{a_{1}} &  X \ar@{=}[d] \\
K'_d \ar[r]_{k'_d} & K'_{d-1} \ar[r]_{k'_{d-1}} &  \cdots \ar[r]_{k'_2} &  K'_{1} \ar[r]_{k'_1} &  X
}
\]
and
\[
\xymatrix{
K'_d\ar[d]^{b_d} \ar[r]^{k'_{d}} & K'_{d-1} \ar[r]^{k'_{d-1}} \ar[d]^{b_{d-1}} &  \cdots \ar[r]^{k'_{2}} &  K'_{1} \ar[r]^{k'_{1}}\ar[d]^{b_{1}} &  X \ar@{=}[d] \\
K_d \ar[r]_{k_d} & K_{d-1} \ar[r]_{k_{d-1}} &  \cdots \ar[r]_{k_2} &  K_{1} \ar[r]_{k_1} &  X.
}
\]
Now consider $c_i = 1_{K_i} - b_i\circ a_i \in \MM(K_i,K_i)$. Using the property of $d$-kernels we find morphisms $h_i$ in the diagram below
\[
\xymatrix{
K_d\ar[d]^{c_d} \ar[r]^{k_{d}} & K_{d-1} \ar[r]^{k_{d-1}} \ar[d]^{c_{d-1}} \ar[dl]^{h_{d-1}} &  \cdots \ar[dl]^{h_{d-2}} \ar[r]^{k_{2}} &  K_{1} \ar[r]^{k_{1}}\ar[d]^{c_{1}} \ar[dl]^{h_{1}} &  X \ar[d]^0 \\
K_d \ar[r]_{k_{d}} & K_{d-1} \ar[r]_{k_{d-1}} &  \cdots \ar[r]_{k_{2}} &  K_{1} \ar[r]_{k_{1}} &  X
}
\]
such that
\[
\begin{cases}
c_1 = k_2 \circ h_1,\\
c_i  = k_{i+1} \circ h_{i} + h_{i-1} \circ k_{i}, & 2 \leqslant i \leqslant d-1,\\
c_d = h_{d-1} \circ k_{d}.
 \end{cases}
\]
In particular, $c_i \in \Rad_\MM(K_i,K_i)$ and so $b_i \circ a_i = 1_{K_i} - c_i$ is an isomorphism. Hence we may write $1_{K_i} = ((b_i \circ a_i)^{-1} b_i) \circ a_i$. The claim follows as $\MM$ has split idempotents.

(2) This is dual to (1).

(3) Existence of $\E$ follows in the same way as in (1). Now let
\[
\E': \quad X \overset{e_{d+1}'}\to E'_{d} \overset{e_d'} \to \cdots \overset{e_2'} \to E'_{1} \overset{e_1'}\to Y
\]
be an equivalent $d$-extension. Then there is a commutative diagram
\[
\xymatrix{
X\ar@{=}[d] \ar[r]^{e_{d+1}} & E_{d} \ar[r]^{e_{d}} \ar[d]^{a_d} &  \cdots \ar[r]^{e_{2}} &  E_{1} \ar[r]^{e_{1}}\ar[d] _{a_1}&  Y \ar@{=}[d] \\
X \ar[r]_{e'_{d+1}} & E'_{d} \ar[r]_{e'_{d}} &  \cdots \ar[r]_{e'_{2}} &  E'_{1} \ar[r]_{e'_{1}} &  Y.
}
\]
By \cite[Proposition 4.10]{J}, there is also a commutative diagram
\[
\xymatrix{
X\ar@{=}[d] \ar[r]^{e'_{d+1}} & E'_{d} \ar[r]^{e'_{d}} \ar[d]^{b_d} &  \cdots \ar[r]^{e'_{2}} &  E'_{1} \ar[r]^{e'_{1}}\ar[d] _{b_1}&  Y \ar@{=}[d] \\
X \ar[r]_{e_{d+1}} & E_{d} \ar[r]_{e_{d}} &  \cdots \ar[r]_{e_{2}} &  E_{1} \ar[r]_{e_{1}} &  Y.
}
\]
As in (1) we proceed to show that $c_i = 1_{E_i} - b_i\circ a_i \in \Rad_\MM(E_i,E_i)$ by constructing a diagram
\[
\xymatrix{
X\ar[d]^0 \ar[r]^{e_{d+1}} & \;E_{d} \ar[r]^{e_{d}} \ar[dl]^{h_{d}} \ar[d]^{c_d} &  \cdots \ar[r]^{e_{2}} \ar[dl]^{h_{d-1}} &  E_{1} \ar[r]^{e_{1}} \ar[dl]^{h_{1}}\ar[d] _{c_1}&  Y \ar[d]^0 \\
X \ar[r]_{e_{d+1}} & \;E_{d} \ar[r]_{e_{d}} &  \cdots \ar[r]_{e_{2}} &  E_{1} \ar[r]_{e_{1}} &  Y
}
\]
such that
\[
\begin{cases}
c_1 = e_2 \circ h_1,\\
c_i  = e_{i+1} \circ h_{i} + h_{i-1} \circ e_{i}, & 2 \leqslant i \leqslant d,\\
0 = h_{d} \circ e_{d+1}.
 \end{cases}
\]
By the defining property of $d$-cokernels, the last condition implies that $h_d$ factors through $e_d$ and so is a radical morphism. Hence $c_i \in \Rad_\MM(E_i,E_i)$ for all $1 \leqslant i \leqslant d$ and the claim follows as in (1).
\end{proof}

\begin{definition}\label{minimal}
We call the $d$-kernels, $d$-cokernels and $d$-extensions appearing in Proposition~\ref{radical} \emph{minimal}. By the Krull-Schmidt property it follows immediately that they are unique up to isomorphism.
\end{definition}
A consequence of Proposition~\ref{radical} is the following characterisation of wide subcategories.

\begin{corollary}\label{cor.wide}
Let $\MM$ be a $d$-abelian Krull-Schmidt category and $\WW \subseteq \MM$ a full subcategory closed under direct sums and direct summands. Then $\WW$ is wide if and only if the following conditions hold.
\begin{enumerate}
\item For any morphism $f: M \to N$ in $\WW$, the terms of its minimal $d$-kernel and $d$-cokernel in $\MM$ lie in $\WW$.
\item For any $X,Y \in \WW$, every minimal $d$-extension of $Y$ by $X$ in $\MM$ has terms in $\WW$. 
\end{enumerate}
\end{corollary}
The following result about $d$-kernels will be useful to compute wide subcategories.

\begin{proposition}\label{kernel}
Let $\MM$ be a $d$-abelian Krull-Schmidt category, $M \in \MM$ indecomposable and $K_d \rightarrow \cdots \rightarrow K_1 \overset{k}\rightarrow M$ be a $d$-kernel of a non-zero morphism $M \overset{f}\rightarrow N$. If $U \in \MM$ is indecomposable and there exists $g \in \MM(U,M) \setminus \Rad^2_\MM(U,M)$ such that $f \circ g = 0$, then $U$ appears as a summand in $K_1$.
\end{proposition}
\begin{proof}
Assume that $U$ does not appear as a summand in $K_1$. Since $f \circ g = 0$ we may write $g = k \circ h$ for some $h \in \MM(U,K_1) = \Rad_\MM(U,K_1)$. Since $f$ is non-zero and $M$ is indecomposable, $k \in \Rad_\MM(K_1,M)$ and so $g \in \Rad^2_\MM(U,M)$, which is a contradiction.
\end{proof}

\subsection{Wide subcategories of $d$-cluster tilting subcategories}

Let $A$ be a finite dimensional algebra and $\mod A$ the category of finitely generated right $A$-modules. We recall the the definition of $d$-cluster tilting from \cite{I2} (see also \cite[Section 2]{HJV}).
\begin{definition}\label{dCT}
Let $\MM$ be a functorially finite subcategory of $\mod A$. We say that $\MM$ is $d$-cluster tilting if
\[\begin{split}
\MM &= \{X \in \mod A \mid \Ext^i_A(X,M) = 0 \mbox{ for all } M \in \MM, 1 \leqslant i \leqslant d-1\} \\
&= \{X \in \mod A \mid \Ext^i_A(M,X) = 0 \mbox{ for all } M \in \MM, 1 \leqslant i \leqslant d-1\}.
\end{split}\]
\end{definition}
As mentioned in the introduction it is shown in \cite{J} that any $d$-cluster tilting subcategory $\MM$ is $d$-abelian. Moreover, the following characterisation of $d$-kernels, $d$-cokernels and $d$-exact sequences in $\MM$ holds.

\begin{proposition}\label{dCTexact}
Let $\MM \subseteq \mod A$ be $d$-cluster tilting. Then $\MM$ is $d$-abelian.
Let $$M_{d+1} \overset f \to M_{d} \to \cdots \to M_{1} \overset g\to M_{0}$$ be a sequence in $\MM$.
\begin{enumerate}
\item The sequence $M_{d+1} \to M_{d} \to \cdots \to M_{1}$ is a $d$-kernel of $g$ if and only if
\[
0 \to M_{d+1} \to M_{d} \to \cdots \to M_{1} \to M_{0}
\]
is exact in $\mod A$.
\item 
The sequence $M_{d} \to \cdots \to M_{1} \to M_{0}$ is a $d$-cokernel of $f$ if and only if
\[
M_{d+1} \to M_{d} \to \cdots \to M_{1} \to M_{0} \to 0
\]
 it is exact in $\mod A$.
\item The sequence $M_{d+1} \to M_{d} \to \cdots \to M_{1} \to M_{0}$ is $d$-exact if and only if
\[
0 \to M_{d+1} \to M_{d} \to \cdots \to M_{1} \to M_{0} \to 0
\]
is exact in $\mod A$.
\item Let $X,Y \in \MM$. There is a bijection from the set of equivalence classes of $d$-extensions of $Y$ by $X$ in $\MM$ to $\Ext_A^d(Y,X)$ that sends the equivalence class of a $d$-extension $$0 \to X \to E_{d} \to \cdots \to E_{1} \to Y \to 0$$ to its Yoneda-class
\[
[0 \to X \to E_{d} \to \cdots \to E_{1} \to Y \to 0].
\]
\end{enumerate}
\end{proposition}
\begin{proof}
The fact that $\MM$ is $d$-abelian is shown in \cite[Theorem 3.16]{J}.

The statements (1), (2), (3) are well-known and easily shown using the fact that $\Ext^i_A(M,M') = 0$ for all $M,M' \in \MM$ and $1 \leqslant i \leqslant d$.

Statement (4) follows from \cite[Proposition A.1]{I1}.
\end{proof}
It will be convenient to construct wide subcategories from several smaller wide subcategories. For this purpose the following result is useful.

\begin{proposition}\label{directsum}
Let $\MM \subseteq \mod A$ be $d$-cluster tilting and $\WW_1, \WW_2 \subseteq \MM$ wide subcategories. If
\[
\Hom_A(X_1,X_2) = \Hom_A(X_2,X_1) = \Ext_A^d(X_1,X_2) =  \Ext_A^d(X_2,X_1) = 0
\]
for all $X_1 \in \WW_1$, $X_2\in \WW_2$. Then $\add\{\WW_1, \WW_2\} \subseteq \MM$ is wide.
\end{proposition}
\begin{proof}
Let $X,Y\in \add\{\WW_1, \WW_2\}$. Then we may write $X = X_1 \oplus X_2$ and $Y = Y_1 \oplus Y_2$ for some $X_1,Y_1\in \WW_1$ and $X_2,Y_2\in \WW_2$

Let $f : X \to Y$ be a morphism. By assumption we may write $f = f_1 \oplus f_2$ for some $f_1 : X_1 \to Y_1$ and $f_2 : X_2 \to Y_2$. Since $\WW_1$ and $\WW_2$ are wide there are $d$-kernels respectively $d$-cokernels of $f_1$ and $f_2$ with terms in $\add\{\WW_1, \WW_2\}$. Taking their direct sums gives a $d$-kernel respectively $d$-cokernel of $f$.

Next consider the natural map
\[
\Ext_A^d(Y_1,X_1) \oplus \Ext_A^d(Y_2,X_2) \to \Ext_A^d(Y,X)
\]
defined by the biadditivity of $\Ext_A^d$. It maps $([\E_1],[\E_2])$ to $[\E_1 \oplus \E_2]$ and is bijective since $\Ext_A^d(Y_1,X_2) =  \Ext_A^d(Y_2,X_1) = 0$. Hence, by Proposition~\ref{dCTexact}, any $d$-extension of $Y$ by $X$ is equivalent to $\E_1 \oplus \E_2$ for some $d$-extensions $\E_1$ in $\WW_1$ and $\E_2$ in $\WW_2$.
\end{proof}

To compute wide subcategories of $d$-cluster tilting subcategories, we will apply the following result.

\begin{theorem}\label{theoremB}\cite[Theorem B]{HJV}
Let $A$ be a finite dimensional algebra and $\MM$ a $d$-cluster tilting subcategory of $\mod A$. Let $\WW \subseteq \MM$ be an additive subcategory. Let $P \in \WW$ be a module and set $B = \End_A(P)$, so that $P$ becomes a $B$-$A$-bimodule. Assume the following: 
\begin{enumerate}[\rm(1)]
\item As an $A$-module $P$ has finite projective dimension. 
\item $\Ext^i_A(P,P) =0$ for all $i \geqslant 1$.
\item Each $W \in \WW$ admits an $\add{P}$-resolution
\[
0 \to P_m \to \cdots \to P_0 \to W \to 0, \quad P_i \in \add{P}.
\]
\item $\Hom_A(P, \WW) \subseteq \mod B$ is $d$-cluster tilting. 
\end{enumerate}
Then $\WW$ is a wide subcategory of $\MM$ and there is an equivalence of categories
\[
- \otimes_{B} P : \Hom_A(P, \WW) \to \WW.
\]
\end{theorem}

\subsection{Higher Auslander algebras of type $A$}\label{Higher Auslander algebras}

Next we introduce a description of the higher Auslander algebras of type $A$ by quivers and relations. Our notation differs slightly from the one in \cite{OT}. 

We begin by introducing some combinatorial data that is needed to make our definitions.

\begin{definition}\label{combinatorics}
Throughout we fix two integers $n$ and $d$ with $n \geqslant 1$ and $d \geqslant 0$. 

\begin{enumerate}[\rm(1)]
\item Set $N_{n,d} :=\{1,\dots, n+d\}$.

\item Set $\V_{n,d} := \{(x_0,x_1,\ldots,x_d)  \in N_{n,d}^{d+1} \mid x_0 < x_1 < \cdots < x_d\}$.

\item For each $k \in N_{n,d}$, define the partial functions 
\[
\sigma_k^{+}: \V_{n,d} \to \V_{n,d} \quad\mbox{and}\quad \sigma_k^{-}  : \V_{n,d} \to \V_{n,d}
\]
by $\sigma_k^{\pm}(x) = y$, where 
\[
y_i= 
\begin{cases}
x_i \pm 1 & \mbox{if } i = k, \\
x_i & \mbox{if } i \neq k,
\end{cases}
\]
whenever such $y\in \V_{n,d}$ exists.

\item Define the relations $\eHom$ and $\eExt$ on $\V_{n,d}$ by
\[
x \eHom y
\quad \mbox {if and only if} \quad
x_0 \leqslant y_0 < x_1 \leqslant y_1 < \cdots < x_d \leqslant y_d,
\]
and
\[
x \eExt y
\quad \mbox {if and only if} \quad
x_0 < y_0 \leqslant x_1 < y_1 \leqslant \cdots \leqslant x_d < y_d.
\]

\item
Let $x,y \in \V_{n,d}$. We say that $x$ and $y$ interlace in case $x \eHom y$, $y \eHom x$, $x \eExt y$ or $y \eExt x$ hold. 

\item 
Let $\X, \Y \subseteq \V_{n,d}$. We say that $\X$ and $\Y$ interlace in case there are $x \in \X$ and $y \in \Y$ such that $x$ and $y$ interlace. 

\item
For $S\subseteq N_{n,d}$ we let $\X_S$ be the set of all $x \in \V_{n,d}$ such that $x_i \in S$ for all $i$. We say that $S,S' \subseteq N_{n,d}$ interlace in case $\X_S$ and $\X_{S'}$ interlace.

\item For $x \in \V_{n,d}$ we set $S_x =\{x_0,\ldots, x_{d}\} \subseteq N_{n,d}$. Note in particular that $\X_{S_x} = \{x\}$.
\end{enumerate}
\end{definition}

Next define a quiver $Q^{n,d}$ with vertices $\V_{n,d}$ and arrows $\alpha^x_{k} : x \to \sigma_k^+(x)$, for all $k \in N_{n,d}$ and $x \in \V_{n,d}$ such that $\sigma_k^+(x)$ is defined. 

Let $k,l \in N_{n,d}$ be two distinct elements and $x\in \V_{n,d}$ be such that $y \in \V_{n,d}$, where
\[
y_i= 
\begin{cases}
x_i + 1 & \mbox{if } i \in \{k,l\}, \\
x_i & \mbox{if } i \not \in \{k,l\}.
\end{cases}
\]
We introduce a relation $\rho^x_{kl}$ from $x$ to $y$ defined by
\[
\rho^x_{kl} = 
\begin{cases}
\alpha^x_{k}\alpha^{\sigma_k^+(x)}_{l} - \alpha^x_{l}\alpha^{\sigma_l^+(x)}_{k} & \mbox{if both $\sigma_k^+(x)$ and $\sigma_l^+(x)$ are defined}, \\
\alpha^x_{k}\alpha^{\sigma_k^+(x)}_{l} & \mbox{if $\sigma_k^+(x)$ is defined but $\sigma_l^+(x)$ is undefined}, \\
\alpha^x_{l}\alpha^{\sigma_l^+(x)}_{k} & \mbox{if $\sigma_l^+(x)$ is defined but $\sigma_k^+(x)$ is undefined},
 \end{cases}
\]
and let $I_{n,d}$ be the ideal in the path algebra $KQ^{n,d}$ generated by all $\rho^x_{kl}$.

The algebra $A_n^d = (KQ^{n,d}/I_{n,d})^{\rm op}$ is called a higher Auslander algebra of type $A$. Note that if $d = 0$, then $A_ n^d$ is just a path algebra of Dynkin type $A_n$, which strictly should not be called a higher Auslander algebra. We include it so that for $d >0$ we may say that $A_n^d$ is the higher Auslander algebra of $A_n^{d-1}$. To explain why this is so we define an indecomposable $A_{n}^{d-1}$-module $M_x$ for each $x \in \V_{n,d}$. As a representation $M_x$ assigns the vector space $K$ to all vertices $y \in \V_{n,d-1}$ such that $x_i \leqslant y_i  < x_{i+1}$ for all $0 \leqslant i \leqslant d-1$. To all other vertices $M_x$ assigns the zero vector space. Moreover, all arrows $K \to K$ act as the identity, while other arrows (by necessity) act as zero.

\begin{example}\label{example.41}
Let us consider $n =4 $ and $d = 2$. Then the quiver $Q^{4,1}$ is
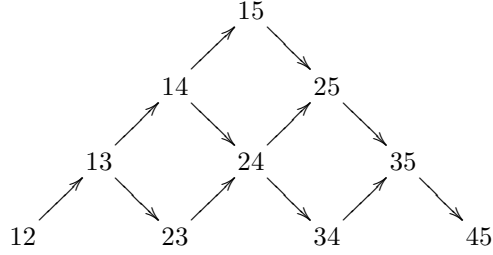
\[
\xy 
(0,0)*+{12}="12";
(10,10)*+{13}="13";
(20,20)*+{14}="14";
(30,30)*+{15}="15";
(20,0)*+{23}="23";
(30,10)*+{24}="24";
(40,20)*+{25}="25";
(40,0)*+{34}="34";
(50,10)*+{35}="35";
(60,0)*+{45}="45";
{\ar^{} "12";"13"};
{\ar^{} "13";"14"};
{\ar^{} "14";"15"};
{\ar^{} "13";"23"};
{\ar^{} "14";"24"};
{\ar^{} "15";"25"};
{\ar^{} "23";"24"};
{\ar^{} "24";"25"};
{\ar^{} "24";"34"};
{\ar^{} "25";"35"};
{\ar^{} "34";"35"};
{\ar^{} "35";"45"};
\endxy
\]
The modules $M_{246}$ and $M_{136}$ are given as representations of $(Q^{4,1})^{\rm op}$ as follows
\[
\xy 
(0,20)*+{M_{246}:}="M";
(0,0)*+{0}="12";
(10,10)*+{0}="13";
(20,20)*+{0}="14";
(30,30)*+{0}="15";
(20,0)*+{0}="23";
(30,10)*+{K}="24";
(40,20)*+{K}="25";
(40,0)*+{K}="34";
(50,10)*+{K}="35";
(60,0)*+{0}="45";
{\ar_{0} "13";"12"};
{\ar_{0} "14";"13"};
{\ar_{0} "15";"14"};
{\ar_{0} "23";"13"};
{\ar_{0} "24";"14"};
{\ar_{0} "25";"15"};
{\ar_{0} "24";"23"};
{\ar_{1} "25";"24"};
{\ar_{1} "34";"24"};
{\ar_{1} "35";"25"};
{\ar_{1} "35";"34"};
{\ar_{0} "45";"35"};
\endxy
\quad \quad
\xy 
(0,20)*+{M_{136}:}="M";
(0,0)*+{0}="12";
(10,10)*+{K}="13";
(20,20)*+{K}="14";
(30,30)*+{K}="15";
(20,0)*+{K}="23";
(30,10)*+{K}="24";
(40,20)*+{K}="25";
(40,0)*+{0}="34";
(50,10)*+{0}="35";
(60,0)*+{0}="45";
{\ar_{0} "13";"12"};
{\ar_{1} "14";"13"};
{\ar_{1} "15";"14"};
{\ar_{1} "23";"13"};
{\ar_{1} "24";"14"};
{\ar_{1} "25";"15"};
{\ar_{1} "24";"23"};
{\ar_{1} "25";"24"};
{\ar_{0} "34";"24"};
{\ar_{0} "35";"25"};
{\ar_{0} "35";"34"};
{\ar_{0} "45";"35"};
\endxy
\]
\end{example}

\begin{theorem}\label{fundamental}\cite{I2}\cite[Section 3]{OT}
Let $n$ and $d$ be positive integers. Then $A_{n}^{d-1}$ has global dimension $d$ and admits a unique basic $d$-cluster tilting module
$$M = \bigoplus_{x \in \V_{n,d}} M_x.$$
Moreover, $\End_{A_{n}^{d-1}}(M)$ is isomorphic to $A_{n}^{d}$.
\end{theorem}

We let $\MM_{n,d} = \add\{M\}$, where $M$ is the $d$-cluster tilting module in Theorem~\ref{fundamental}. Then $\MM_{n,d}$ is $d$-abelian. Our aim is to classify the wide subcategories of $\MM_{n,d}$.

The isomorphism from $A_{n}^{d}$ to $\End_{A_{n}^{d-1}}(M)$ comes from realising the quiver of $\MM_{n,d}$ as $Q^{n,d}$. We illustrate this in one example.
\begin{example}\label{example.42}
Consider the case $n=4$ and $d = 2$. Below is the quiver of $\MM_{4,2}$ (i.e, $Q^{4,2}$).

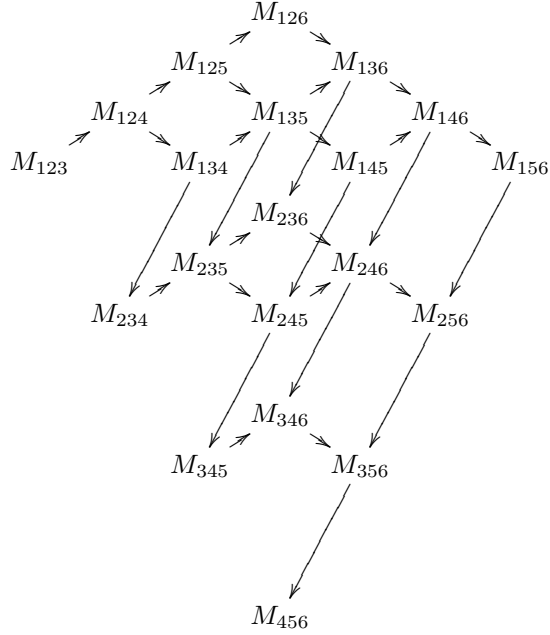
\[
\xy 0;<3pt,0pt>:<0pt,1.9pt>:: 
(0,0)*+{M_{123}}="123";
(10,10)*+{M_{124}}="124";
(20,20)*+{M_{125}}="125";
(30,30)*+{M_{126}}="126";
(20,0)*+{M_{134}}="134";
(30,10)*+{M_{135}}="135";
(40,20)*+{M_{136}}="136";
(40,0)*+{M_{145}}="145";
(50,10)*+{M_{146}}="146";
(60,0)*+{M_{156}}="156";
(10,-30)*+{M_{234}}="234";
(20,-20)*+{M_{235}}="235";
(30,-10)*+{M_{236}}="236";
(30,-30)*+{M_{245}}="245";
(40,-20)*+{M_{246}}="246";
(50,-30)*+{M_{256}}="256";
(20,-60)*+{M_{345}}="345";
(30,-50)*+{M_{346}}="346";
(40,-60)*+{M_{356}}="356";
(30,-90)*+{M_{456}}="456";
{\ar^{} "123";"124"};
{\ar^{} "124";"125"};
{\ar^{} "125";"126"};
{\ar^{} "124";"134"};
{\ar^{} "125";"135"};
{\ar^{} "126";"136"};
{\ar^{} "134";"135"};
{\ar^{} "135";"136"};
{\ar^{} "135";"145"};
{\ar^{} "136";"146"};
{\ar^{} "145";"146"};
{\ar^{} "146";"156"};
{\ar^{} "134";"234"};
{\ar^{} "135";"235"};
{\ar^{} "136";"236"};
{\ar^{} "145";"245"};
{\ar^{} "146";"246"};
{\ar^{} "156";"256"};
{\ar^{} "234";"235"};
{\ar^{} "235";"236"};
{\ar^{} "235";"245"};
{\ar^{} "236";"246"};
{\ar^{} "245";"246"};
{\ar^{} "246";"256"};
{\ar^{} "245";"345"};
{\ar^{} "246";"346"};
{\ar^{} "256";"356"};
{\ar^{} "345";"346"};
{\ar^{} "346";"356"};
{\ar^{} "356";"456"};
\endxy
\]

Note in particular, that there is a path from $M_{136}$ to $M_{246}$, the modules appearing in Example~\ref{example.41}. Considering the relations defining $A_{4}^{2}$, this path should correspond to a nonzero morphism $\phi : M_{136} \to M_{246}$. Indeed, such a $\phi$ is easy to find. As a morphism of representations we may define it by $\phi_{24} = \phi_{25} = 1_K$ and $\phi_{ij} = 0$ for all other indices $ij$. Compare this with the fact that $(1,3,6) \eHom (2,4,6)$.
\end{example}

As in the above example, there is in general an obvious bijection between the indecomposables in $\MM_{n,d}$ and $Q^{n,d}_0$. Moreover, morphisms corresponding to the arrows in $Q^{n,d}_0$ are easy to write down. In fact morphisms and extensions between indecomposables in $\MM_{n,d}$ have been computed in \cite{OT}. Here we recall some of their results rewritten in our notation.

\begin{theorem}\label{HomAndExt}\cite[Theorem 3.6]{OT}
\begin{enumerate}[\rm(1)]
\item Let $x \in \V_{n,d-1}$ and $x' = (1,x_0+1,\ldots,x_{d-1}+1) \in \V_{n,d}$. Then $e_x A_{n}^{d-1} = M_{x'}$.
\item Let $y \in \V_{n,d-1}$ and $y' = (y_0,\ldots,y_{d-1},n+d) \in \V_{n,d}$. Then $D(A_{n}^{d-1}e_y) = M_{y'}$.
\item For $x,y \in \V_{n,d}$ we have
\[
\dim_K\Hom_{A_{n}^{d-1}}(M_x,M_y) = \begin{cases} 1 & \mbox{if } x \eHom y, \\ 0 & \mbox{else.} \end{cases}
\]
\item For $x,y \in \V_{n,d}$ we have
\[
\dim_K\Ext^d_{A_{n}^{d-1}}(M_y,M_x) = \begin{cases} 1 & \mbox{if } x \eExt y, \\ 0 & \mbox{else.} \end{cases}
\]
\end{enumerate}
\end{theorem}

Note that in particular Theorem \ref{HomAndExt}(3) implies that $\End_{A_{n}^{d-1}}(M_x) = K$, which verifies that $M_x$ is indeed indecomposable. 

By Theorem \ref{HomAndExt}(3) we find that if $x \eHom y$, then any path from $x$ to $y$ in $Q^{n,d}$ gives a basis of $\Hom_{A_{n}^{d-1}}(M_x,M_y)$. Similarly, Theorem \ref{HomAndExt}(4) can be made more explicit. 

\begin{theorem}\label{theorem.extensions}\cite[Theorem 3.8]{OT}
If $x \eExt y$, then there is an exact sequence 
\[
\E_{xy} : \quad 0 \to M_x \to E_d \to \cdots E_1 \to M_y \to 0,
\]
where $E_k = \bigoplus_{z} M_{z}$, taken over all $z \in \V_{n,d}$ such that $z_i \in \{x_i,y_i\}$ for all $i$ and $$|\{i \mid z_i = x_i\}| = k.$$ 
\end{theorem}

Note that $x \eExt y$ implies $x_i \neq y_i$ for all $i$, and so $|\{i \mid z_i = x_i\}| = k$ may be replaced with $$|\{i \mid z_i = y_i\}| = d-k+1$$ in the above condition. Moreover, we may deduce the following result which is useful for computing wide subcategories

\begin{proposition}\label{extensionClosed}
If $x \eExt y$, then $M_z \in \wide\{M_x,M_y\}$ for any $z \in \V_{n,d}$ that satisfies $z_i \in \{x_i,y_i\}$ for all $i$.
\end{proposition}
\begin{proof}
Each module $M_z$ satisfying $z_i \in \{x_i,y_i\}$ for all $i$, appears as a summand in exactly one term of $\E_{xy}$. Hence the morphisms in $\E_{xy}$ are all radical morphisms. The claim now follows from Corollary~\ref{cor.wide} and Proposition~\ref{dCTexact}.
\end{proof}

\begin{corollary}\label{resolutions}
Let $x\in \V_{n,d}$ and $s \in N_{n,d}$ such that $s < x_0$. Then there is an exact sequence
\[
0 \to M_{x^d} \to \cdots \to M_{x^0} \to M_x \to 0
\]
where $x^i = (s,x_0,\ldots, x_{i-1}, x_{i+1},\ldots, x_d) \in \V_{n,d}$. In particular, if $s = 1$, this is the minimal projective resolution of $M_x$.
\end{corollary}
\begin{proof}
The sequence is precisely $\E_{x^dx}$ from Theorem~\ref{theorem.extensions}. If $s = 1$, then each $M_{x^i}$ is projective by Theorem~\ref{HomAndExt}(1).
\end{proof}

\section{Wide subcategories}
\subsection{Main result}
In this section we classify wide subcategories of $\MM_{n,d}$. We start by introducing the basic building blocks for such subcategories. Let $S \subseteq N_{n,d}$ and recall the subset $\X_S \subseteq \V_{n,d}$ from Definition~\ref{combinatorics}(7). We set
\[
\WW_S = \add\{M_x \mid x \in \X_S\} \subseteq \MM_{n,d}.
\]
Since $\X_{S_x} = \{x\}$ (see Definition~\ref{combinatorics}(8)), we have in particular that $\WW_{S_x} = \add\{M_x\}$.

Note that $\WW_S$ is non-zero if and only if $|S| \geqslant d+1$. In that case we call $S$ admissible. In the interest of brevity a set of admissible subsets of $N_{n,d}$ is called a collection. A collection is called non-interlacing if it has no two distinct members that interlace. With this terminology we are now ready to state the main result of our paper.

\begin{theorem}\label{main}
Let $n$ and $d$ be positive integers. Then there is a bijection 
\[
\{\mbox{non-interlacing collections of subsets of } N_{n,d}\} \rightarrow \{\mbox{wide subcategories of } \MM_{n,d}\}
\]
that sends a collection $\Sigma$ to $\add\{\WW_S \mid S \in \Sigma\}$.
\end{theorem}

\begin{remark}\label{interlace}
Note that by Theorem~\ref{HomAndExt}, two admissible subsets $S, S' \subseteq N_{n,d}$ interlace if and only if there are $M_x \in \WW_S$ and $M_{x'} \in \WW_{S'}$ such that one of the following conditions hold:
\[
\Hom_{A_n^{d-1}}(M_x, M_{x'}) \neq 0, \quad
\Ext^d_{A_n^{d-1}}(M_x, M_{x'}) \neq 0,
\]
\[
\Hom_{A_n^{d-1}}(M_{x'}, M_{x}) \neq 0, \quad
\Ext^d_{A_n^{d-1}}(M_{x'}, M_{x}) \neq 0.
\]
\end{remark}
In the remaining sections we will prove Theorem~\ref{main}. 

\subsection{Injectivity}
In this section we show that the map in Theorem~\ref{main} is well-defined and injective. We begin by showing that categories of the form $\WW_S$ are wide.

\begin{proposition}\label{prop.wide}
Let $S \subseteq N_{n,d}$ be admissible. Then
\begin{enumerate}[\rm(1)]
\item The subcategory $\WW_S \subseteq \MM_{n,d}$ is wide. 
\item Let $n' = |S|-d$ and $\iota : N_{n',d} \to S$ be the unique order preserving bijection. Then there is an equivalence from $\MM_{n',d}$ to $\WW_S$ sending $M_{x'}$ to $M_{\iota(x')}$, where $\iota(x') = (\iota(x_0'),\ldots, \iota(x_d'))$.
\item In particular, the terms of $d$-kernels, $d$-cokernels and $d$-extensions can be computed for $\WW_S$ as in $\MM_{n',d}$ using $\iota$.
\end{enumerate}

\end{proposition}
\begin{proof}
We apply Theorem~\ref{theoremB}. Let $s = \min S$. Set $$P = \bigoplus_{x \in \X_S,\, x_0 = s} M_x.$$ Since $A_n^{d-1}$ has global dimension $d$ the projective dimension of $P$ is at most $d$. Theorem~\ref{HomAndExt}(4) implies $\Ext^d_{A_n^{d-1}}(P,P) = 0 $. The exact sequence in Corollary~\ref{resolutions} provides an $\add\{P\}$-resolution for each $M_x$, with $x \in \X_S$. 

It remains to show that Theorem~\ref{theoremB}(4) holds in this case. To this end, we first claim that $\End_{A_n^{d-1}}(P)$ is isomorphic to $A_{n'}^{d-1}$. To see this note that $\iota$ gives a bijection between the the vertices of $Q^{n',d-1}$ and the indecomposable summands of $P$ by sending $(x'_0,\ldots, x'_{d-1})$ to $M_x$, where $$x = (s,\iota(x'_0+1),\iota(x'_1+1),\ldots, \iota(x'_{d-1}+1)).$$ Given the descriptions of $A_{n'}^{d-1}$ and $\MM_{n,d}$ by quivers and relations, it is readily checked that this bijection extends to an isomorphism from $A_{n'}^{d-1}$ to $\End_{A_n^{d-1}}(P)$.

Hence $\mod \End_{A_n^{d-1}}(P)$ has a unique $d$-cluster tilting subcategory that we may identify with $\MM_{n',d}$. Next we claim that under this identification $\Hom_{A_{n'}^{d-1}}(P,-)$ sends $M_x \in \WW_S$ to $M_{\iota^{-1}(x)}$. If $x_0 = s$, then $M_x \in \add \{P\}$ and the claim is immediate. Otherwise the consider the $\add\{P\}$-resolution of $M_x$ given in Corollary~\ref{resolutions}. Applying $\Hom_{A_{n'}^{d-1}}(P,-)$ to this resolution we find the minimal projective resolution of $M_{\iota^{-1}(x)} \in \MM_{n',d}$ and the claim follows. Hence $\Hom_{A_{n'}^{d-1}}(P,\WW_S)$ is the $d$-cluster tilting subcategory of $\mod \End_{A_n^{d-1}}(P)$. 

Thus Theorem~\ref{theoremB} applies and the equivalence claimed in part (2) is given by the functor $-\otimes_{\End_{A_n^{d-1}}(P)} P$.
\end{proof}

\begin{corollary}\label{injective}
The map in Theorem~\ref{main} is well-defined and injective.
\end{corollary}
\begin{proof}
For a non-interlacing collection $\Sigma$ we need to show that $\add\{\WW_S \mid S \in \Sigma\}$ is wide. By Proposition~\ref{prop.wide}(1), each $\WW_S$ is wide. Moreover, since the sets $S$ are non-interlacing, there are no non-trivial morphisms or $d$-extensions between modules in $\WW_S$ and $\WW_{S'}$ for $S \neq S'$ (see Remark~\ref{interlace}). It follows from Proposition~\ref{directsum} that $\add\{\WW_S \mid S \in \Sigma\}$ is wide.

To show injectivity consider a wide subcategory $\WW$ in the image of the map in Theorem~\ref{main}. Let $\mathcal{P}_\WW$ be the poset of all admissible sets $S$ satisfying $\WW_S \subseteq \WW$ ordered by inclusion. Now write $\WW = \add\{\WW_S \mid S \in \Sigma\}$ for a non-interlacing collection $\Sigma = \{S_1, \ldots, S_l\}$. We claim that $\Sigma$ equals the set of maximal elements in $\mathcal{P}_\WW$. This implies injectivity as we can recover $\Sigma$ from $\WW$. 

To show the claim first note that for $i \neq j$ we have $S_i \not \subseteq S_j$ as $S_i$ and $S_j$ do not interlace. It remains to show for all $S \in \mathcal{P}_\WW$ that $S \subseteq S_i$ for some $i$. To do this note that $\WW_S$ is connected and satisfies
\[
\WW_S \subseteq \add\{\WW_{S_1}, \ldots, \WW_{S_l}\}.
\]
Since $\Sigma$ is non-interlacing there is no walk in $\WW$ connecting some $M_x \in \WW_{S_i}$ with some $M_y \in \WW_{S_j}$ for $i \neq j$, and so $\WW_S \subseteq \WW_{S_i}$ for some $i$, which implies $S \subseteq S_i$.
\end{proof}

\subsection{Surjectivity}
It remains to show that the map in Theorem~\ref{main} is surjective. We will do this using several Lemmas, each stating that a certain wide subcategory is of the form $\WW_S$ for some admissible set $S$. Lemmas \ref{start}, \ref{simples} and \ref{2close} take care of certain special cases that have to be dealt with separately.

\begin{lemma}\label{start}
Assume $n \geqslant 2$. Let $l \geqslant d+2$ and $s_1 < \cdots < s_l$ be elements in $N_{n,d}$. Set $S = \{s_1,\ldots, s_{l-1}\}$ and $S' = \{s_2,\ldots, s_l\}$. Then $$\wide\{\WW_S,\WW_{S'}\} = \WW_{S\cup S'}.$$
\end{lemma}
\begin{proof}
We apply Proposition~\ref{prop.wide}(3) to $S\cup S'$. Hence we may assume that $S = \{1,\ldots, n+d-1\}$ and $S' = \{2,\ldots, n+d\}$ so that $\WW_{S\cup S'} = \MM_{n,d}$. By Proposition~\ref{extensionClosed} any module $M_z \in \MM_{n,d}$ satisfies $M_z \in \wide\{M_x,M_y\}$ for some $M_x \in \WW_S$ and $M_y \in \WW_{S'}$. The claim follows. 
\end{proof}

\begin{lemma}\label{simples}
We have
\[
\wide\{M_{(1,2,\ldots, d+1)}, \WW_{\{2,3,\ldots,n+d\}}\} = \MM_{n,d}.
\]
\end{lemma}
\begin{proof}
The proof is by induction on $n$. The case $n = 1$ is trivial. Assume $n > 1$. By induction hypothesis 
\[
\wide\{M_{(1,2,\ldots, d+1)}, \WW_{\{2,3,\ldots,n+d-1\}}\} = \WW_{\{1,2,\ldots,n+d-1\}}.
\]
Since $\WW_{\{2,3,\ldots,n+d-1\}} \subseteq \WW_{\{2,3,\ldots,n+d\}}$ we have 
\[
\wide\{M_{(1,2,\ldots, d+1)}, \WW_{\{2,3,\ldots,n+d\}}\} \supseteq \wide\{\WW_{\{1,2,\ldots,n+d-1\}},\WW_{\{2,3,\ldots,n+d\}}\}.
\]
Now applying Lemma~\ref{start} with $S = \{1,2,\ldots,n+d-1\}$ and $S' = \{2,3,\ldots,n+d\}$ we get 
\[
\wide\{M_{(1,2,\ldots, d+1)}, \WW_{\{2,3,\ldots,n+d\}}\} \supseteq \WW_{S\cup S'} = \MM_{n,d}.
\]
\end{proof}

\begin{lemma}\label{2close}
Let $x,x' \in \V_{n,d}$ such that $x_k \neq x'_k$ for some $k$ and $x_i = x'_i$ for all $i \neq k$. Then
\[
\wide\{M_x,M_{x'}\} = \WW_{S_x\cup S_{x'}}.
\]
\end{lemma}
\begin{proof}
Without loss of generality assume $x_k < x'_k$. Apply Corollary~\ref{resolutions} to 
\[y = (x_1,x_2,\ldots, x_k,x'_k,x_{k+1},\ldots, x_d)\]
with $s = x_0$ to obtain the exact sequence
\[
\E : \quad 0 \to M_{y^d} \to \cdots \to M_{y^0} \to M_y \to 0.
\]
Then $y^k = x$ and $y^{k-1} = x'$ (for the extreme case $k = 0$, this should be interpreted as $y = y^{-1} = x'$).

Extracting the morphism $M_{y^k} \to M_{y^{k-1}}$ we get a non-zero morphism $f: M_{x'} \to M_x$. Moreover, we may view the sequence $\E$ as the concatenation of the map $f$ with its minimal $d$-kernel and $d$-cokernel. In particular, the terms of the sequence all belong to $\wide\{M_x,M_{x'}\}$. But the terms are exactly the indecomposable modules in $\WW_{S_x\cup S_{x'}}$.
\end{proof}

In general the $d$-kernels and $d$-cokernels of morphisms $M_x \to M_y$ are not as easy to compute as in the proof of Lemma~\ref{2close}. Fortunately we will not need complete information about such $d$-kernels and $d$-cokernels to prove our results. More precisely we will make do with the following statement.

\begin{lemma}\label{kernel2}
Let $f : M_x \to M_y$ be a non-zero morphism.
\begin{enumerate}[\rm(1)]
\item If $x_{k-1} < y_{k-1} < x_{k}$, then for $x' = (x_0,\ldots,x_{k-1},y_{k-1},x_{k+1},\ldots,x_{d})$ the module $M_{x'}$ appears as a summand in the term $K_1$ in any $d$-kernel $K_d \rightarrow \cdots \rightarrow K_1 \rightarrow M_x$ of $f$. In particular, $M_{x'} \in \wide\{M_x,M_y\}$ and there is a non-zero morphism $M_{x'} \to M_x$.
\item If $y_k < x_{k+1} < y_{k+1}$, then for $y' = (y_0,\ldots,y_{k-1},x_{k+1},y_{k+1},\ldots,y_{d})$ the module $M_{y'}$ appears as a summand in the term $C_1$ of any $d$-cokernel $M_y \rightarrow C_1 \rightarrow \cdots \rightarrow C_d$ of $f$. In particular, $M_{y'} \in \wide\{M_x,M_y\}$ and there is a non-zero morphism $M_{y} \to M_{y'}$.
\end{enumerate}
\end{lemma}
\begin{proof}
We only prove (1) as (2) is dual.

The $d$-kernels of $f$ may be computed in the wide subcategory $\WW_{S}$, where $S = S_x \cup S_y$. By Proposition~\ref{prop.wide}(3), we may assume that $S = N_{n,d}$ and $\WW_{S} = \MM_{n,d}$. Then $x' = \sigma_k^-(x)$ and so there is an arrow in $Q_n^d$ from $x'$ to $x$. Hence there is a corresponding morphism $g \in \Rad_{\MM_{n,d}}(M_{x'},M_x) \setminus \Rad^2_{\MM_{n,d}}(M_{x'},M_x)$. Notice that $x' \noteHom y$ and so $\Hom_{A_n^{d-1}}(M_{x'},M_y) = 0$ by Theorem~\ref{HomAndExt}(3). In particular, $f\circ g = 0$ and the claim follows from Proposition~\ref{kernel}.
\end{proof}

Before continuing with the proof we illustrate the utility of Lemma~\ref{kernel2} in an example.
\begin{example}\label{example.wide}
Let $n = 4$ and $d = 2$ as in Example~\ref{example.42} and consider $\WW = \wide\{M_{136}, M_{246}\}$. Since $\{1,3,6\}$ and $\{2,4,6\}$ interlace, we must have $\WW = \WW_{\{1,2,3,4,6\}}$ in order for Theorem~\ref{main} to hold. Indeed, if $\WW = \add\{\WW_S \mid S \in \Sigma\}$ for some non-interlacing collection $\Sigma$, we must have $\{1,3,6\} \subseteq S$ and $\{2,4,6\} \subseteq S$ for some common $S \in \Sigma$ and so the smallest possible collection is $\Sigma = \{\{1,2,3,4,6\}\}$. Below is the quiver of $\MM_{4,2}$ with the indecomposables in $\WW_{\{1,2,3,4,6\}}$ underlined.

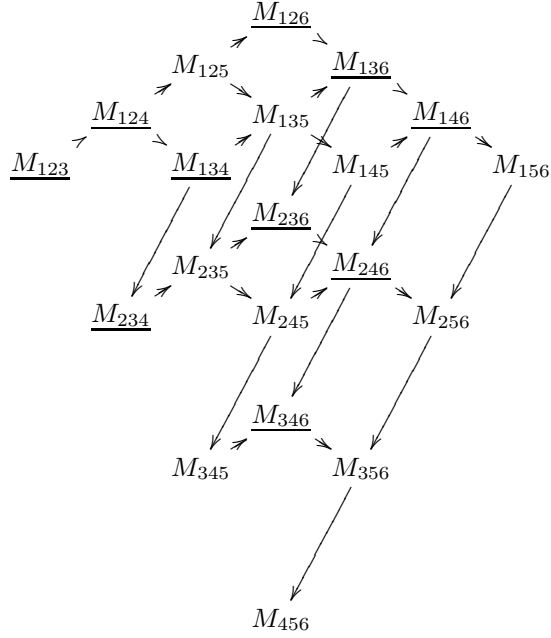
\[
\xy 0;<3pt, 0pt>:<0pt,1.9pt>:: 
(0,0)*+{\underline{M_{123}}}="123";
(10,10)*+{\underline{M_{124}}}="124";
(20,20)*+{M_{125}}="125";
(30,30)*+{\underline{M_{126}}}="126";
(20,0)*+{\underline{M_{134}}}="134";
(30,10)*+{M_{135}}="135";
(40,20)*+{\underline{M_{136}}}="136";
(40,0)*+{M_{145}}="145";
(50,10)*+{\underline{M_{146}}}="146";
(60,0)*+{M_{156}}="156";
(10,-30)*+{\underline{M_{234}}}="234";
(20,-20)*+{M_{235}}="235";
(30,-10)*+{\underline{M_{236}}}="236";
(30,-30)*+{M_{245}}="245";
(40,-20)*+{\underline{M_{246}}}="246";
(50,-30)*+{M_{256}}="256";
(20,-60)*+{M_{345}}="345";
(30,-50)*+{\underline{M_{346}}}="346";
(40,-60)*+{M_{356}}="356";
(30,-90)*+{M_{456}}="456";
{\ar^{} "123";"124"};
{\ar^{} "124";"125"};
{\ar^{} "125";"126"};
{\ar^{} "124";"134"};
{\ar^{} "125";"135"};
{\ar^{} "126";"136"};
{\ar^{} "134";"135"};
{\ar^{} "135";"136"};
{\ar^{} "135";"145"};
{\ar^{} "136";"146"};
{\ar^{} "145";"146"};
{\ar^{} "146";"156"};
{\ar^{} "134";"234"};
{\ar^{} "135";"235"};
{\ar^{} "136";"236"};
{\ar^{} "145";"245"};
{\ar^{} "146";"246"};
{\ar^{} "156";"256"};
{\ar^{} "234";"235"};
{\ar^{} "235";"236"};
{\ar^{} "235";"245"};
{\ar^{} "236";"246"};
{\ar^{} "245";"246"};
{\ar^{} "246";"256"};
{\ar^{} "245";"345"};
{\ar^{} "246";"346"};
{\ar^{} "256";"356"};
{\ar^{} "345";"346"};
{\ar^{} "346";"356"};
{\ar^{} "356";"456"};
\endxy
\]

We verify that $\WW = \WW_{\{1,2,3,4,6\}}$ does indeed hold by applying Lemma~\ref{kernel2} repeatedly. 

Considering the $2$-kernel of the nonzero morphism $M_{136} \to M_{246}$ we find that $M_{126}, M_{134} \in \WW$. Considering the $2$-cokernel we find $M_{346} \in \WW$. Similarly, the $2$-cokernel of the the nonzero morphism $M_{134} \to M_{136}$ gives $M_{146} \in \WW$.

Next we consider the nonzero morphism $M_{126} \to M_{136}$ and get $M_{123}, M_{236} \in \WW$. This allows us to consider $M_{236} \to M_{246}$, which gives $M_{234} \in \WW$, and then $M_{134} \to M_{234}$, which gives $M_{124} \in \WW$. 

Hence all $10$ indecomposables in $\WW_{\{1,2,3,4,6\}}$ lie in $\WW$. 
\end{example}

We now continue with the general case.
The following Lemma is the main tool in the proof of Theorem~\ref{main} and can be thought of as a generalisation of Example~\ref{example.wide}. The proof is similar to the strategy of Example~\ref{example.wide} in that we successively build up more and more objects in a certain wide subcategory. 

\begin{lemma}\label{grow}
Let $x \in \V_{n,d}$ and $S \subseteq N_{n,d}$ admissible. If 
\[
\Hom_{A_n^{d-1}}(\WW_S, M_x) \neq 0 \quad \mbox{or} \quad
\Hom_{A_n^{d-1}}(M_x, \WW_S) \neq 0,
\]
then $\wide\{\WW_S, M_x\} = \WW_{S\cup S_x}$.
\end{lemma}
\begin{proof}
Recall from Theorem~\ref{HomAndExt}(3) that
\[
\dim_K\Hom_{A_{n}^{d-1}}(M_x,M_y) = \begin{cases} 1 & \mbox{if } x \eHom y, \\ 0 & \mbox{else.} \end{cases}
\]
We will use this freely throughout the proof to characterise when $\Hom_{A_{n}^{d-1}}(M_x,M_y) = 0$.

Set $\WW = \wide\{\WW_S, M_x\}$. It is enough to show the inclusion $\WW_{S\cup S_x} \subseteq \WW$. Note that the claim is trivial if $S_x \subseteq S$. We proceed in several steps.

Step 1: First consider the case when $S_x$ contains only one element $a$ that is not in $S$, i.e., $S\cup S_x = S\cup \{a\}$. We use induction on $|S|$. If $|S| = d+1$, then $\WW_S$ consists of a single indecomposable $M_{x'}$ and the claim follows by Lemma~\ref{2close}.

Next assume that $|S| > d+1$. By Proposition~\ref{prop.wide}(3), we may assume $S\cup \{a\} = N_{n,d}$ so that $\WW_{S\cup S_x} = \MM_{n,d}$. We proceed with a case by case analysis depending on the value of $a$. Our strategy in each case is to produce more and more elements of $\WW$ to finally reach $\MM_{n,d}.$

Step 1.1: Assume $a = 1$, i.e., $S = \{2,\ldots, n+d\}$. Then $x_0 =1$ and we have $\Hom_{A_n^{d-1}}(\WW_S,M_x) = 0$ so that
\[
\Hom_{A_n^{d-1}}(M_x, \WW_S) \neq 0
\]
which implies $x_1\geqslant 3$. Next let $x'\in \V_{n,d}$ be such that $x'_0 = 1$, $x'_1\geqslant 3$, $M_{x'} \in \WW$ and the value of $\sum_{i=1}^d x'_i$ is as small as possible. We claim that $$x' = (1,3,4,\ldots,d+2).$$ Otherwise there is some $b\geqslant 3$ satisfying $x'_{k-1} < b < x'_{k}$ for some $k \geqslant 1$. Set 
\[y = \begin{cases}
(2,x'_1,\ldots,x'_{k-2},b,x'_{k},\ldots, x'_d), & \mbox{if } k\geqslant 2, \\
(b,x'_1,\ldots, x'_d), & \mbox{if } k = 1.
\end{cases}\]
Then $M_y \in \WW_S$ and applying Lemma~\ref{kernel2}(1) to a non-zero morphism $M_{x'} \to M_y$ we get that $M_{x''} \in \wide\{M_{x'},M_y\} \subseteq \WW$ for $$x'' = (1,x_1',\ldots, x_{k-1}',b,x_{k+1}',\ldots,x_d')$$ contradicting the minimality of $\sum_{i=1}^d x'_i$.

Next set $$y' = (2,3,\ldots, d+2).$$ Then $M_{y'} \in \WW$ and applying Lemma~\ref{2close} to $M_{x'}$ and $M_{y'}$ we find that
\[
M_{(1,2\ldots,d+1)} \in \wide\{M_{x'},M_{y'}\} \subseteq \WW.
\]
Finally by Lemma~\ref{simples},
\[
\MM_{n,d} = \wide\{M_{(1,2,\ldots, d+1)}, \WW_{\{2,3,\ldots,n+d\}}\} \subseteq \WW.
\]

Step 1.2: Assume $a = n+d$, i.e., $S = \{1,\ldots, n+d-1\}$. This is similar to Step 1.1 and the proof is omitted.

Step 1.3: Assume $1 < a < n+d$, i.e., $S = \{1,\ldots,a-1,a+1,\ldots, n+d\}$. Then $x_k = a$ for some $k$. The condition
\[
\Hom_{A_n^{d-1}}(\WW_S, M_x) \neq 0\, \quad \mbox{or} \quad \Hom_{A_n^{d-1}}(M_x, \WW_S) \neq 0,
\]
implies that $x_{k-1} < a-1$ or $x_{k+1} > a+1$. We treat these inequalities in three separate substeps.

Step 1.3.1: Assume both $x_{k-1} < a-1$ and $x_{k+1} > a+1$ hold. In particular $1 \leqslant k \leqslant d-1$. Set
\[
x' = (x_0,\ldots,x_{k-1},a-1,x_{k+1},\ldots,x_d).
\]
Then $M_{x'} \in \WW_S \subseteq \WW$ and applying Lemma~\ref{2close} to $M_x$ and $M_{x'}$ we find that $M_{x''} \in \WW$ for
\[
x'' = (x_1,\ldots, x_{k-1},a-1,a,x_{k+1},\ldots, x_d).
\]
Set
\[
y = (x_1,\ldots, x_{k-1},a-1,a+1,x_{k+1},\ldots, x_d)
\]
and $S' = S \setminus\{1\}$. Then $M_{x''} \in \WW_{S' \cup\{ a\}}$ and $M_y \in \WW_{S'}$. Moreover, $\Hom_{A_n^{d-1}}(M_{x''} ,M_{y})\neq 0$ and so by induction hypothesis
\[
\wide\{M_{x''}, \WW_{S'}\} = \WW_{S' \cup\{ a\}} = \WW_{\{2,\ldots, n+d\}}.
\]
Hence
\[
\WW \supseteq \wide\{M_{x''}, \WW_{S'}, \WW_{S}\} =  \wide\{\WW_{\{2,\ldots, n+d\}}, \WW_{S}\}.
\]
Set
\[
z = 
 \begin{cases}
(1,3,\ldots,a-1,a+1,\ldots,d+3), & \mbox{if } a\leqslant d+2, \\
(1,3,\ldots,d+2), & \mbox{if } a > d+2.
\end{cases}
\]
Then $M_z \in \WW_S$ and $\Hom_{A_n^{d-1}}(M_{z},\WW_{\{2,\ldots, n+d\}})\neq 0$. By Step 1.1
\[
\wide\{\WW_{\{2,\ldots, n+d\}}, \WW_{S}\} = \MM_{n,d}.
\]

Step 1.3.2: Assume $x_{k-1} < a-1$ and $x_{k+1} = a+1$. Since $|S| > d+1$ there is some $b \in S\setminus (S_x \cup \{a-1\})$. Choose $b$ as small as possible. There are two cases to consider: $b < a-1$ and $b > a+1$.

Step 1.3.2.1: Assume $b < a-1$.  First we replace $x$ with something more suitable. For this purpose let $x' \in \V_{n,d}$ be such that $x'_i = x_i$ for all $i \geqslant k$, $x'_{k-1} < a-1$, $M_{x'} \in \WW$ and $\sum_{i = 1}^{k-1} x'_{i}$ as large as possible. Notice that such $x'$ exists since $x$ is a candidate. We claim that $x'_0 > 1$. Otherwise there is $b' < a-1$ such that $x'_l < b' < x'_{l+1}$ for some $0 \leqslant l \leqslant k-1$. Set
\[
y = 
 \begin{cases}
(x'_0,\ldots,x'_l,b',x'_{l+2},\ldots,x'_{k-1}, a-1,x'_{k+1},\ldots, x'_d), & \mbox{if } l< k-1, \\
(x'_0,\ldots,x'_l,b',x'_{k+1},\ldots, x'_d), & \mbox{if } l = k-1.
\end{cases}
\]
Then $M_{y} \in \WW_S \subseteq \WW$ and applying Lemma~\ref{kernel2}(2) to a non-zero morphism $M_y \to M_{x'}$ we find that $M_{x''} \in \WW$ for
\[
x'' = (x'_0,\ldots,x'_{l-1},b',x'_{l+1},\ldots, x'_d)
\] 
contradicting the maximality of $\sum_{i = 1}^{k-1} x'_{i}$.

Now set
\[
y' = (x'_0,\ldots,x'_{k-1}, a-1, x'_{k+1},\ldots,x'_{d}).
\]
Then $\Hom_{A_n^{d-1}}(M_{y'},M_{x'})\neq 0$ and $M_{y'} \in \WW_{S'}$ for $S' = S \setminus\{1\}$. As in step $1.3.1$ we get by induction that
\[
\wide\{M_{x'}, \WW_{S'}\} = \WW_{S' \cup\{ a\}} = \WW_{\{2,\ldots, n+d\}}
\]
and Step 1.1 yields
\[
\WW \supseteq \wide\{M_{x'}, \WW_{S'}, \WW_{S}\} = \MM_{n,d}.
\]

Step 1.3.2.2: Assume $b > a+1$. As before we replace $x$ with something more suitable. Let $x' \in \V_{n,d}$ be such that $x'_i = x_i$ for all $i \leqslant k$, $M_{x'} \in \WW$ and $\sum_{i = k+1}^d x'_{i}$ as large as possible. Note that there is a $b'$ such that $x'_l < b' < x'_{l+1}$ for some $l \geqslant k$ or $x'_d < b'$.

We claim that $l = k$ so that Step 1.3.1 can be applied after replacing $x$ by $x'$. To show this we assume $l > k$ and reach a contradiction. 

Let
\[
x'' = (x'_0,\ldots, x'_{k-1},a-1,x'_{k+1},\ldots, x'_d)
\]
and apply Lemma~\ref{2close} to $M_{x'}$ and $M_{x''}$ to obtain $M_y \in \WW$ for
\[
y = (x'_0,\ldots, x'_{k-1},a-1,a,x'_{k+2},\ldots, x'_d).
\]
Next set
\[
y' = (x'_0,\ldots, x'_{k-1},a-1,x'_{k+1},x'_{k+2},\ldots, x'_{l-1}, b', x'_{l+1},\ldots, x'_d).
\] 
Then $M_{y'} \in \WW_S \subseteq \WW$. Applying Lemma~\ref{kernel2}(2) to a non-zero morphism $M_y \to M_{y'}$ we obtain $M_{y''} \in \WW$ for
\[
y'' = (x'_0,\ldots, x'_{k-1},a,x'_{k+1},x'_{k+2},\ldots, x'_{l-1}, b', x'_{l+1},\ldots, x'_d).
\]
But then $\sum_{i = k+1}^d y''_{i} > \sum_{i = k+1}^d x'_{i}$, which is a contradiction.

Step 1.3.3: Assume $x_{k-1} = a-1$ and $x_{k+1} > a+1$. This is similar to step 1.3.2 and the proof is omitted.

Step 2: Now consider the general case. We proceed by induction on $m = |S_x\setminus S|$. Note that the cases $m = 0$ and $m= 1$ have already been proved. Thus consider the case $m > 1$. We assume $\Hom_{A_n^{d-1}}(M_x, \WW_S) \neq 0$ (the case $\Hom_{A_n^{d-1}}(\WW_S, M_x) \neq 0$ is similar). Then there exists $y$ with $S_y \subseteq S$ such that there is a non-zero morphism $f : M_x \to M_y$. Hence $x \eHom y$. Since $m > 1$ there must be a $1 \leqslant k \leqslant d$ such that $x_{k-1} \not \in S$. In particular $y_{k-1} \neq x_{k-1}$ and so $$x_{k-1} < y_{k-1} < x_{k}.$$ Hence by Lemma~\ref{kernel2}(1), we get that $M_{x'} \in \WW$ for $$x' = (x_0,\ldots,x_{k-1},y_{k-1},x_{k+1},\ldots,x_{d}).$$ Moreover, there is a non-zero morphism $g : M_{x'} \to M_x$. Now $x_{k-1} < y_{k-1} < x_{k}$ also means that $x_{k-1} < x'_{k} < x_{k}$ so by Lemma~\ref{kernel2}(2) we find that $M_{x''} \in \WW$ for $$x'' = (x_0,\ldots,x_{k-2},y_{k-1},x_{k},x_{k+1},\ldots,x_{d}).$$ 
Notice that $x''$ is obtained from $x$ by replacing $x_{k-1}$ (which is not in $S$) by $y_{k-1}$ (which is in $S$). In particular, $|S_{x''}\setminus S| < m$ and setting $S' = S\cup S_{x''}$ we get by induction that $\WW_{S'} \subseteq \wide\{\WW_S, M_{x''}\}$ and so $\WW_{S'} \subseteq \WW$. On the other hand $\Hom_{A_n^{d-1}}(M_x, \WW_{S'}) \neq 0$ and $S_x \setminus S' = \{x_{k-1}\}$ so by Step 1 we get that $\WW_{S'\cup S_x} \subseteq \wide\{\WW_{S'}, M_{x}\} \subseteq \WW$. But $S' \cup S_x = S\cup S_x$ so the claim follows.
\end{proof}

\begin{lemma}\label{maximalDisjoint}
Let $\WW$ be a wide subcategory, $M_x \in \WW$ and $\WW_S \subseteq \WW$ for some admissible set $S$ that is maximal (with respect to inclusion) with this property. If there is $M_y \in \WW_S$ such that there is a walk of morphisms in $\WW$ connecting $M_x$ and $M_y$ then $M_x \in \WW_S$. 
\end{lemma}
\begin{proof}
The claim follows in case the walk consists of a single morphism by maximality and applying Lemma~\ref{grow}. Propagating along an arbitrary walk gives the general result.
\end{proof}

\begin{proposition}\label{surjective}
The map in Theorem~\ref{main} is surjective.
\end{proposition}
\begin{proof}
Let $\WW$ be a wide subcategory. Consider again the poset $\mathcal{P}_\WW$ of all admissible subsets $S$ satisfying $\WW_S \subseteq \WW$ ordered by inclusion. Let $\Sigma$ be the collection of maximal elements in $\mathcal{P}_\WW$. Since every module $M_x \in \WW$ lies in some $\WW_S \subseteq \WW$ (e.g., $S = S_x$) it follows that $\WW = \add\{\WW_S \mid S \in \Sigma\}$. It remains to show that $\Sigma$ is non-interlacing. We do this by showing that if $S,S' \in \Sigma$ interlace, then $\WW_S = \WW_{S'}$ and so $S = S'$. Note that if $S,S' \in \Sigma$ interlace, then by Remark~\ref{interlace}, there are $M_x \in \WW_S$ and $M_{x'} \in \WW_{S'}$ such that there is either a non-zero morphism connecting $M_x$ and $M_{x'}$ or a non-trivial $d$-extension in $\WW$ with endpoints $M_x$ and $M_{x'}$ (in some order) as described in Theorem~\ref{theorem.extensions}. In either case there is a walk in $\WW$ from $M_x$ to $M_{x'}$. Moreover, since $\WW_S$ and $\WW_{S'}$ are connected there is in fact a walk in $\WW$ from any indecomposable in $\WW_S$ to any indecomposable in $\WW_{S'}$. The claim now follows by Lemma~\ref{maximalDisjoint}.
\end{proof}

This completes the proof of Theorem~\ref{main}.

\subsection{Number of wide subcategories}
Finally, we briefly discuss the number of wide subcategories of $\MM_{n,d}$, which we denote by $w_{n,d}$. For $d = 1$ it is well-known that the numbers $w_{n,1}$ are Catalan numbers:
\[
w_{n,1} = \frac{1}{n+2} {{2n+2}\choose{n+1}}.
\]
For $n = 1$ we have $w_{1,d} = 2$ since $\MM_{n,d}$ has precisely one indecomposable in this case.

For $n = 2$ the category $\MM_{n,d}$ has $d+2$ indecomposables
\[
M_1 = M_{1,2,\ldots, d,d+1}, \;
M_2 = M_{1,2,\ldots, d,d+2}, \ldots, \;
M_{d+1} = M_{1,3,\ldots, d+1,d+2}, \;
M_{d+2} = M_{2,3,\ldots, d+1,d+2}.
\]
Hence a wide subcategory of $\MM_{n,d}$ is determined by a subset of $\{M_1,\ldots, M_{d+2}\}$, which is naturally encoded as word of length $d+2$ in letters $0,1$. Such a word corresponds to a wide subcategory if and only if it avoids any occurrence of $11$ when read cyclically or is the word $11\ldots 1$. Counting the number of such words is straightforward. In fact, $(w_{2,d})_d$ appears as A001612 in \cite{OEIS} and satisfies the recurrence relation
\[
w_{2,d} = w_{2,d-1} + w_{2,d-2} -1. 
\]

For general $n$ and $d$, we have not found an easy formula $w_{n,d}$. However, due to the simple combinatorial description of wide subcategories in Theoreom~\ref{main} it is straightforward to write an algorithm that computes the number $w_{n,d}$ for any $n \geqslant 1$, $d \geqslant 1$. For instance one may generate each $\Sigma$ by iteratively attaching admissible sets $S$ that do not interlace. Running such an algorithm on a computer one finds the following values for small $n$ and $d$.
\[
\begin{array}{c}
\mbox{Some values of $w_{n,d}$} \\
\begin{array}{|r|r|r|r|r|r|r|r|r|}
\hline
d	&w_{1,d}	&w_{2,d}	&w_{3,d}	&w_{4,d}	&w_{5,d}	&w_{6,d}	&w_{7,d}	&w_{8,d} \\ \hline
1	&2	&5	&14	&42	&132&	429	&1,430	&4,862 \\ 
2	&2	&8	&47	&374	&4,083	&62,824	&1,376,012	&42,579,642 \\ 
3	&2	&12	&237	&16,830	&4,597,078 &&&\\ 
4	&2	&19	&1,724	&3,499,884 &&&& \\ 
5	&2	&30	&17,934 &&&&&\\ 
6	&2	&48	&273,092 &&&&&\\ 
7	&2	&77	&5,732,137 &&&&&\\ \hline
\end{array}
\end{array}
\]


\begin{thebibliography}{MMSS}

\bibitem[ASS]{ASS}
I.\ Assem, D.\ Simson and A.\ Skowro\'{n}ski,
\emph{Elements of the representation theory of associative algebras},
Vol.\ 1, London Math.\ Soc.\ Stud.\ Texts, Vol.\ 65, Cambridge
University Press, Cambridge, 2006.

\bibitem[BIK]{BIK}  D.\ Benson, S.\ Iyengar, and H.\ Krause, \emph{Local cohomology and support for triangulated categories}, Ann.\ Sci.\ \'{E}c.\ Norm.\ Sup\'{e}r.\ (4) {\bf 41} (2008), 573--619.

\bibitem[Br]{Br}  N.\ T.\ Broomhead, \emph{Thick subcategories of discrete derived categories}, Adv.\ Math.\ {\bf 336} (2018), 242--298. 

\bibitem[B]{B} K.\ {Br{\"u}ning},
\emph{Thick subcategories of the derived category of a hereditary algebra}, 
Homology Homotopy Appl.\ \textbf{9} (2007) 165--176.

\bibitem[DJL]{DJL} T.\ Dyckerhoff, G.\ Jasso and Y.\ Lekili,
\emph{The symplectic geometry of higher Auslander algebras: Symmetric products of disks}, preprint (2019).  {\tt arXiv:1911.11719.}

\bibitem[F]{F} F.\ Fedele, 
\emph{Auslander-Reiten $(d+2)$-angles in subcategories and a $(d+2)$-angulated generalisation of a theorem by Br{\"u}ning}, 
J.\ Pure Appl.\ Algebra \textbf{223} (2019), 3554--3580.

\bibitem[GG]{GG}  M.\ Garcia and A.\ Garver, \emph{Semistable subcategories for tiling algebras}, Beitr.\ Algebra Geom.\ {\bf 61} (2020), 47--71.

\bibitem[GKO]{GKO} C.\ Geiss, B.\ Keller and S.\ Oppermann,
\emph{$n$-angulated categories}, 
J.\ Reine Angew.\ Math.\ \textbf{675} (2013) 101--120.

\bibitem[GS]{GS}  S.\ Gratz and G.\ Stevenson, \emph{On the graded dual numbers, arcs, and non-crossing partitions of the integers}, J.\ Algebra {\bf 515} (2018), 360--388. 

\bibitem[HJV]{HJV} M.\ Herschend, P.\ J{\o}rgensen and L.\ Vaso
\emph{Wide subcategories of $d$-cluster tilting subcategories}, Trans.\ Amer.\ Math.\ Soc.\ {\bf 373} (2020), 2281--2309.

\bibitem[H]{H}  M.\ Hovey, \emph{Classifying subcategories of modules}, Trans.\ Amer.\ Math.\ Soc.\ {\bf 353} (2001), 3181--3191.

\bibitem[HK]{HK}  A.\ Hubery and H.\ Krause, \emph{A categorification of non-crossing partitions}, J.\ Eur.\ Math.\ Soc.\ (JEMS) {\bf 18} (2016), 2273--2313. 

\bibitem[IPT]{IPT}  C.\ Ingalls, C.\ Paquette, and H.\ Thomas, \emph{Semi-stable subcategories for Euclidean quivers}, Proc.\ London Math.\ Soc.\ (3) {\bf 110} (2015), 805--840.

\bibitem[IT]{IT}  C.\ Ingalls and H.\ Thomas, 
{\em Noncrossing partitions and representations of quivers}, 
Compositio Math.\ {\bf 145} (2009), 1533--1562.

\bibitem[I2]{I2} O.\ Iyama, \emph{Cluster tilting for higher Auslander algebras},
Adv.\ Math.\ {\bf 226} (2011), 1--61.

\bibitem[I1]{I1} O.\ Iyama, \emph{Higher-dimensional Auslander-Reiten theory on maximal orthogonal subcategories}, 
Adv.\ Math.\ {\bf 210} (2007), 22--50.

\bibitem[IO]{IO} O.\ Iyama and S.\ Oppermann, \emph{$n$-representation-finite algebras and $n$-APR tilting}, Trans.\ Amer.\ Math.\ Soc.\ {\bf 363} (2011), 6575--6614.

\bibitem[JJ]{JJ} K.\ M.\ Jacobsen and P.\ J\o rgensen, \emph{$d$-abelian quotients of $(d+2)$-angulated categories},
J.\ Algebra \textbf{521} (2019), 114--136.

\bibitem[J]{J} G.\ Jasso, 
\emph{$n$-abelian and $n$-exact categories},
Math.\ Z.\ {\bf 283} (2016), 703--759.

\bibitem[MMSS]{MMSS}  E.\ N.\ Marcos, O.\ Mendoza, C.\ S\'{a}enz, and V.\ Santiago, \emph{Wide subcategories of finitely generated $\Lambda$-modules}, J.\ Algebra Appl.\ {\bf 17} (2018), no.\ 5, 1850082.

\bibitem[MS]{MS}  F.\ Marks and J.\ \v{S}\v{t}ov\'{\i}\v{c}ek, 
{\em Torsion classes, wide subcategories and localisations}, 
Bull.\ London Math.\ Soc.\ {\bf 49} (2017), 405--416.

\bibitem[OEIS]{OEIS}
OEIS Foundation Inc.\ (2020), The On-Line Encyclopedia of Integer Sequences, http://oeis.org/A001612

\bibitem[OT]{OT} S.\ Oppermann and H.\ Thomas, 
\emph{Higher-dimensional cluster combinatorics and representation theory}, 
J.\ Eur.\ Math.\ Soc.\ (JEMS) {\bf 14} (2012), 1679--1737

\bibitem[R]{R}  C.\ M.\ Ringel, \emph{Bricks in hereditary length categories}, Results Math.\ {\bf 6} (1983), 64--70.

\bibitem[Y]{Y}  T.\ Yurikusa, \emph{Wide subcategories are semistable}, 
Doc.\ Math.\ {\bf 23} (2018), 35--47. 

\end{thebibliography}
\end{document}